\documentclass[a4paper]{amsart}

\title[Non-singular and pmp actions of infinite permutation groups]
  {Non-singular and probability measure-preserving actions of infinite permutation groups}

\author{Todor Tsankov}
\address{
  Universit\'e Claude Bernard Lyon 1 \\
  Institut Camille Jordan \\
  Lyon \\
  France
  -- and --
  Institut Universitaire de France}
\email{tsankov@math.univ-lyon1.fr}

\date{November 2024} 

\subjclass[2020]{Primary 37A40,37A50; Secondary 03C15}
\keywords{non-singular actions, oligomorphic groups, de Finetti theorem, primitive actions}

\usepackage[T1]{fontenc}
\usepackage[osf]{mathpazo}

\usepackage{eucal}

\usepackage{microtype}

\usepackage{hyperref}
\usepackage[initials, shortalphabetic]{amsrefs}

\usepackage{my-macros}

\usepackage{enumitem}
\setlist[enumerate,1]{label=(\roman*), font=\normalfont}

\numberwithin{equation}{section}

\newcommand{\RaNi}[2]{\frac{\ud #1}{\ud #2}}

\DeclareMathOperator{\Fin}{Fin}

\begin{document}

\begin{abstract}
  We prove two theorems in the ergodic theory of infinite permutation groups.
First, generalizing a theorem of Nessonov for the infinite symmetric group, we show that every non-singular action of a non-archimedean, Roelcke precompact, Polish group on a measure space $(\Omega, \mu)$ admits an invariant $\sigma$-finite measure equivalent to $\mu$. Second, we prove the following de Finetti type theorem: if $G \curvearrowright M$ is a primitive permutation group with no algebraicity verifying an additional uniformity assumption, which is automatically satisfied if $G$ is Roelcke precompact, then any $G$-invariant, ergodic probability measure on $Z^M$, where $Z$ is a Polish space, is a product measure.
\end{abstract}

\maketitle

\section{Introduction}

The theory of dynamical systems of Polish (non-locally compact) groups has recently seen a rapid development, as many connections with combinatorics and probability theory have emerged. This paper is a contribution to the ergodic theory of infinite permutation groups, that is, dynamical systems that preserve a measure or a measure class.

Infinite permutation groups arise naturally in a model-theoretic context as the automorphism groups of countable structures. A particularly important class of such groups, where the action remembers all model-theoretic information and  there is a perfect dictionary between model theory and permutation group theory, is that of oligomorphic groups. Recall that a permutation group $G \leq \Sym(M)$ is called \df{oligomorphic} if the diagonal action $G \actson M^n$ has finitely many orbits for every $n$. From a model-theoretic perspective, these are precisely the automorphism groups of $\aleph_0$-categorical structures. If one slightly relaxes this condition and considers \df{locally oligomorphic} actions instead (an action is \df{locally oligomorphic} if its restriction to any finitely many orbits is oligomorphic), there is also a characterization in terms of topological groups that makes no mention of a particular action. Recall that a topological group $G$ is \df{non-archimedean} if it admits a basis at $1_G$ consisting of open subgroups and it is \df{Roelcke precompact} if for every neighborhood $U$ of $1_G$, there is a finite set $F \sub G$ such that $G = UFU$. Then a non-archimedean Polish group is Roelcke precompact iff it admits a continuous, topologically faithful, locally oligomorphic action on a discrete countable set iff all of its continuous actions on a discrete set are locally oligomorphic \cite{Tsankov2012}. An action of $G$ on a discrete set $M$ is \df{topologically faithful} if the associated homomorphism $G \to \Sym(M)$ is a homeomorphic embedding, where $\Sym(M)$ is equipped with the pointwise convergence topology. Some (uninteresting, from a dynamical perspective) examples of Roelcke precompact groups that cannot be represented by faithful oligomorphic actions are the infinite compact non-archimedean groups, i.e., the profinite groups.

Roelcke precompact groups share some tameness properties with compact groups. For example, the unitary representations of non-archimedean Roelcke precompact groups can be completely classified \cite{Tsankov2012}. However, in contrast to the situation with compact groups, irreducible representations are usually infinite-dimensional. This classification has proved helpful for studying their probability measure-preserving (\df{pmp}, for short) actions. For example, an analogue of de Finetti's theorem holds for appropriate actions of such groups (\cite{Jahel2022}, also see below) and Jahel and Joseph~\cite{Jahel2025} have been able to classify the \emph{invariant random subgroups} of some of them.

The first part of this paper deals with non-singular actions of non-archimedean, Roelcke precompact, Polish groups. A \df{non-singular action} is an action on a $\sigma$-finite measure space that preserves the measure class; equivalently, it is an action on the measure algebra by automorphisms, not necessarily preserving the measure. This is a classical topic in ergodic theory with many applications. In a recent development connected to permutation groups, Conley, Jahel and Panagiotopoulos~\cite{Conley2024p}, extending a result of Ackerman, Freer, and Patel~\cite{Ackerman2016} for invariant measures, described all subgroups $H \leq \Sym(\N)$ such that the homogeneous space $\Sym(\N)/H$ carries a measure for which the translation action is non-singular. In a somewhat different direction, Neretin~\cite{Neretin2024} also studies non-singular actions of some Polish groups.

Examples of non-singular actions can be obtained by starting with a system that preserves a finite or a $\sigma$-finite measure and then taking a measure in the same measure class. It turns out that for non-archimedean, Roelcke precompact, Polish groups, this is all one can do, as our first main theorem shows.

\begin{theorem}
  \label{th:intro:non-singular}
  Let $G$ be a non-archimedean, Roelcke precompact, Polish group and let $G \actson (\Omega, \mu)$ be a non-singular action. Then $\Omega$ can be represented as a disjoint countable union $\bigsqcup_i \Omega_i$ of invariant subsets on each of which the action is isomorphic to an induced action from a pmp action of an open subgroup of $G$. In particular, there exists a $\sigma$-finite measure equivalent to $\mu$ which is preserved by $G$.
\end{theorem}

Induction is a standard construction in ergodic theory that takes an action $V \actson X$ of an open subgroup $V \leq G$ and produces a suitable action of $G$ on the disjoint union of $[G : V]$ copies of $X$. See \autoref{sec:induced-actions} for more details.

A slightly less detailed version of this theorem, in the special case where $G$ is the full symmetric group $\Sym(\N)$, was obtained by Nessonov~\cite{Nessonov2020}.

It is instructive to compare \autoref{th:intro:non-singular} with the situation for compact groups. An ergodic, non-singular action of a compact group $G$ is conjugate to the action on a homogeneous space $G \actson G/H$ equipped with the quotient of the Haar measure, which is finite, while an ergodic action of a Roelcke precompact, non-archimedean group can preserve an infinite measure (for example, any transitive action on a countably infinite set). Moreover, pmp actions of such groups can be very complicated.

The proof of the theorem is based on the classification of the unitary representations of such groups as well as on a new result concerning the algebraic closure operator for locally oligomorphic actions: in model-theoretic terms, the quasi-order on $M\eq$ given by $a \in \acl\eq b$ is well-founded (\autoref{th:A-partial-order}).

\autoref{th:intro:non-singular} reduces the study of non-singular actions of a group $G$ satisfying its hypotheses to pmp actions of its open subgroups, making the question of understanding the pmp actions of these groups all the more interesting. This problem is intimately connected to \df{exchangeability theory} in probability, which, in most general terms, attempts to classify all possible distributions of random variables $(\xi_a : a \in M)$ invariant under the action of a certain permutation group $G \leq \Sym(M)$. This is equivalent to classifying the invariant measures for the shift action $G \actson Z^M$, where, without loss of generality, one can take $Z$ to be an arbitrary fixed uncountable Polish space, for example, the interval $[0, 1]$. In view of the ergodic decomposition theorem, one can also assume that the measure is ergodic. The problem is easier the bigger the group $G$ is (as there are fewer invariant measures), and indeed, the first result of the theory is the classical de Finetti theorem which states that the only such ergodic measures are the product measures when $G$ is the full symmetric group $\Sym(M)$. In probabilistic terms, this means that the variables $\xi_a$ are independent, identically distributed. While a full classification with a reasonably general hypothesis on the permutation group $G \leq \Sym(M)$ (for example, oligomorphic) seems for the moment out of reach, in the second part of the paper, we concentrate on isolating the optimal hypotheses to obtain de Finetti's conclusion of independence.

This problem was already considered in \cite{Jahel2022}*{Theorem~1.1}, where it was proved that the conclusion holds if $G \leq \Sym(M)$ is a transitive oligomorphic permutation group with no algebraicity and weak elimination of imaginaries, and the techniques used there were again based on the classification of the unitary representations. The \df{no algebraicity} assumption (this means that the actions of stabilizers of finite subsets $A \sub M$ have infinite orbits outside of $A$) is clearly necessary; see \cite{Jahel2022}*{Example~5.1.2}. It is natural to try to replace weak elimination of imaginaries by the weaker assumption of primitivity of the action but it is claimed in \cite{Jahel2022}*{Example~5.1.3} that this is not possible. However, the example contains a mistake (cf. \autoref{rem:JT-example}) and primitivity is indeed sufficient, as follows from the theorem below. Finally, the assumption of the action being oligomorphic is essential for the methods of \cite{Jahel2022} because a classification of the unitary representations is not available for non-Roelcke precompact groups, but here we present a different proof based entirely on probability theory that dispenses with this assumption. The price we have to pay is an additional uniformity requirement in either of the no algebraicity or the primitivity hypothesis. The precise definitions are given in \autoref{sec:some-gener-deFinetti}.

\begin{theorem}
  \label{th:intro:deFinetti}
  Let $G \leq \Sym(M)$ be a primitive permutation group with no algebraicity such that
  at least one of the following two conditions holds:
  \begin{enumerate}
  \item \label{i:th:intro:una} $G \actson M$ has uniform non-algebraicity; or
  \item \label{i:th:intro:bded-diam} Every stabilizer $G_a$ is a boundedly maximal subgroup of $G$.
  \end{enumerate}
  (If $G$ is Roelcke precompact, then both conditions hold automatically under the other hypotheses.)
  
  Then any $G$-invariant, ergodic measure on $[0, 1]^M$ is a product measure.
\end{theorem}
Condition \autoref{i:th:intro:bded-diam} above is satisfied, for example, if the action $G \actson M^2$ has finitely many orbits but also in many other situations. This condition has been considered in the pseudo-finite case in \cite{Liebeck2010}.
Both conditions \autoref{i:th:intro:una} and \autoref{i:th:intro:bded-diam} are remnants of compactness and are automatically satisfied if $G$ is Roelcke precompact. Condition \autoref{i:th:intro:una} is also satisfied, for example,
if $M$ is a homogeneous structure whose age is given by a first-order universal theory (or, equivalently, a collection of  \emph{finite} forbidden configurations) and $G = \Aut(M)$. An important class of examples where the hypotheses are not satisfied but the conclusion still holds is the rational Urysohn space and some generalizations of it. In a very recent preprint \cite{Barritault2024pp}, Barritault, Jahel, and Joseph manage to classify all unitary representations of its isometry group and deduce, in particular, a version of \autoref{th:intro:deFinetti}.

Jahel and Perruchaud \cite{Jahel2024p} have constructed an example of a primitive action $G \actson M$ with no algebraicity, and an invariant measure on $\R^M$ based on a Gaussian random process, where the conclusion of the theorem does not hold, so the uniformity conditions cannot be omitted. Finally, we note that the transitivity assumption is not essential; see \autoref{th:tail-equal-invariant} for the precise general statement of the theorem.

The proof of \autoref{th:intro:deFinetti} is based on a simple, general result that we believe to be of independent interest (cf. \autoref{c:tail-single} and \autoref{rem:tail0-instead-of-tail}).
\begin{theorem}
  \label{th:intro:indep-tail}
 Let $G \leq \Sym(M)$ be a permutation group such that for every $a \in M$, the orbits of $G_a$ on $M \sminus \set{a}$ are infinite. Let $(\xi_a: a \in M)$ be random variables whose joint distribution is $G$-invariant. Then they are conditionally independent over the tail $\sigma$-field.  
\end{theorem}

\subsection*{Acknowledgments}
I would like to thank Tomás Ibarlucía and Yuri Neretin for bringing to my attention Nessonov's article \cite{Nessonov2020}, which inspired \autoref{th:intro:non-singular}. I am also grateful to Rémi Barritault, Colin Jahel, and Matthieu Joseph for useful remarks on a preliminary version of the paper and to the anonymous referee for useful suggestions.


\section{Locally oligomorphic permutation groups}
\label{sec:locally-oligomorphic}

\subsection{The lattice of algebraically closed sets}
\label{sec:latt-algebr-clos}

Let $G$ be a non-archimedean, Roelcke precompact, Polish group. All continuous actions of $G$ on a discrete set $N$ are \df{locally oligomorphic}: i.e., for every $n$ and every orbit $G \cdot a$, there are only finitely many orbits of the diagonal action $G \actson (G \cdot a)^n$ \cite{Tsankov2012}*{Theorem~2.4}. Equivalently, for any tuple $\bar a$ from $N$, every $G$-orbit splits into finitely many $G_{\bar a}$ orbits (here $G_{\bar a}$ denotes the stabilizer of the tuple $\bar a$). It follows that if $G$ acts locally oligomorphically on $N$, then so does any of its open subgroups. The theory of locally oligomorphic permutation groups is parallel to the one of $\aleph_0$-categorical structures in model theory, with the additional subtlety that one must allow infinitely many sorts. Some of the theory was developed using permutation group theoretic language in \cite{Evans2016}. In this section, we prove a well-foundedness result for a certain partial order that will be crucial for the proof of \autoref{th:intro:non-singular}.

Let $G$ be Polish, Roelcke precompact and let $G \actson M$ be a continuous action on the countable set $M$. If $D \sub M$, we denote by $G_D$ the \df{setwise stabilizer} of $D$
\begin{equation*}
  G_D \coloneqq \set{g \in G : g \cdot D = D},
\end{equation*}
and by $G_{(D)}$ the \df{pointwise stabilizer}:
\begin{equation*}
  G_{(D)} \coloneqq \set{g \in G : g \cdot a = a \text{ for all } a \in D}.
\end{equation*}
Borrowing some terminology from model theory, we will say that a set $D \sub M$ is \df{definable} if $G_D$ is open. Note that every finite set is definable.

We denote by $\Fin(M)$ the collection of finite subsets of $M$. If $A \in \Fin(M)$, we define the \df{algebraic closure} of $A$ by:
\begin{equation*}
  \acl A = \set{b \in M : G_A \cdot b \text{ is finite}}.
\end{equation*}
A subset $D \sub M$ is \df{algebraically closed} if $\acl A \sub D$ for every finite $A \sub D$.
A subset $D \sub M$ is \df{locally finite} if it intersects each $G$-orbit in a finite set.

\begin{lemma}
  \label{l:acl}
  Let $A \sub M$ be finite. Then $\acl A$ is definable, algebraically closed, and locally finite.
\end{lemma}
\begin{proof}
  We have that $G_A \leq G_{\acl A}$, so $G_{\acl A}$ is open and $\acl A$ is definable. To see that $\acl A$ is algebraically closed, let $B \sub \acl A$ be finite and let $c \in \acl B$. Then $G_B \cdot c$ is finite, so $(G_A \cap G_B) \cdot c$ is also finite. On the other hand, $G_A \cdot B$ is finite, so $G_A \cap G_B$ has finite index in $G_A$, and therefore $G_A \cdot c$ is also finite. Finally, local finiteness follows from the fact that the action $G_A \actson G \cdot c$ has only finitely many orbits for any $c \in M$, and thus the union of the finite orbits is finite.
\end{proof}

For the next lemma, recall that two subgroups of $G$ are \df{commensurate} if their intersection has finite index in both. If $H \leq G$, the \df{commensurator of $H$ in $G$}, denoted by $\Comm_G(H)$, is the subgroup of all $g \in G$ such that $H$ and $gHg^{-1}$ are commensurate. We also denote by $N_G(H)$ the normalizer of $H$ in $G$. The subgroup $H$ is called \df{self-commensurating} if $\Comm_G(H) = H$ and \df{self-normalizing} if $N_G(H) = H$. 
\begin{lemma}
  \label{l:stab-loc-fin}
  Let $D \sub M$ be definable, algebraically closed, and locally finite. Then the following hold:
  \begin{enumerate}
  \item \label{i:slf:1} There exists a finite $A \sub D$ such that $G_{A} = G_{D}$ and $D = \acl A$. 
  \item \label{i:slf:2} For every $g \in G$, if $g G_D g^{-1} \cap G_D$ has finite index in $G_D$, then $g \in G_D$.
  \item \label{i:slf:3} In particular, $G_{D}$ is self-commensurating and, therefore, self-normalizing.
  \end{enumerate}
\end{lemma}
\begin{proof}
  \autoref{i:slf:1}
  Write $H = G_{D}$. Let $S_0, S_1, \ldots$ be an enumeration of the orbits of $G$ on $M$. Let $A_i = \bigcup_{j \leq i} S_j \cap D$. As $D$ is locally finite, each $A_i$ is finite. Moreover, $G_{A_0} \supseteq G_{A_1} \supseteq \cdots$ and $\bigcap_i G_{A_i} = H$. As $H$ is open, by Roelcke precompactness, each $G_{A_i}$ contains only finitely many $H$-double cosets, so there is $i$ such that $G_{A_i} = H$. Let $A = A_i$. We check that $\acl A = D$. The $\sub$ inclusion follows from the fact that $D$ is algebraically closed. For the other, if $a \in D$, then $G_A \cdot a = H \cdot a \sub D \cap G \cdot a$, which is finite by the local finiteness of $D$.

  \autoref{i:slf:2}
  If $g G_A g^{-1} \cap G_A$ has finite index in $G_A$, then $G_A g \cdot A$ is finite. Thus $g \cdot A \sub \acl A = D$, which implies that $g \cdot A \sub A$. As $A$ is finite, this means that $g \cdot A = A$ and thus $g \in G_A = G_D$.

  \autoref{i:slf:3} This follows directly from \autoref{i:slf:2}.
\end{proof}

\begin{lemma}
  \label{l:pointwise-st}
  Let $D \sub M$ be definable, algebraically closed, and locally finite and let $H = G_{(D)}$. Then:
  \begin{enumerate}
  \item \label{i:ps:1} Every infinite $G_D$-orbit in $\Fin(M)$ splits into infinite $H$-orbits;
  \item \label{i:ps:2} Every $G$-orbit in $\Fin(M)$ splits into finitely many $H$-orbits.
  \end{enumerate}
\end{lemma}
\begin{proof}
  \autoref{i:ps:1} This follows from \cite{Evans2016}*{Lemma~2.4}, noting that for $K \leq G$ and $A \in \Fin(M)$, the orbit $K \cdot A$ is infinite iff there exists $a \in A$ with $K \cdot a$ infinite.

  \autoref{i:ps:2}
  Note that $H \lhd G_D$ and that $G_D/H$ acts faithfully on $D$ with finite orbits. It follows from \cite{Evans2016}*{Lemma~2.9} that the action $G_D/H \actson D$ is topologically faithful, so $G_D/H$ is a profinite group. As every $G$-orbit splits into finitely many $G_D$-orbits by Roelcke precompactness, it suffices to show that every $G_D$-orbit splits into finitely many $H$-orbits. Let $A \in \Fin(M)$ and $F \sub G_D$ be finite such that $G_D = HF(G_A \cap G_D)$. Then $F \cdot A$ is a set of representatives for all $H$-orbits in $G_D \cdot A$.
\end{proof}

Consider the quasi-order $\leq$ on $\Fin(M)$ defined by:
\begin{equation}
  \label{eq:quasi-order}
  A \leq B \iff A \sub \acl B \iff \acl A \sub \acl B \iff G_B \cdot A \text{ is finite}.
\end{equation}
The second equivalence follows from \autoref{l:acl}. 

This quasi-order gives rise to the equivalence relation $\sim$ on $\Fin(M)$ defined by:
\begin{align*}
  A \sim B \iff A \leq B \And B \leq A \iff \acl A = \acl B.
\end{align*}
At the level of stabilizers, $A \leq B$ iff $G_A \cap G_B$ has finite index in $G_B$, and $A \sim B$ iff $G_A$ and $G_B$ are commensurate. It follows from \autoref{eq:quasi-order} that for all $A, B \in \Fin(M)$,
\begin{equation}
  \label{eq:sub-implies-leq}
  A \sub B \implies A \leq B.
\end{equation}
For $A \in \Fin(M)$, we denote by $[A]$ the $\sim$-equivalence class of $A$.
If $A, B \in \Fin(M)$, we write $A < B$ if $A \leq B$ and $A \nsim B$.
\begin{theorem}
  \label{th:A-partial-order}
  Let $G$ be a non-archimedean, Roelcke precompact, Polish group and let $G \actson M$ be a continuous action on a discrete, countable set. Let the quasi-order $\leq$ be defined as in \autoref{eq:quasi-order}. Then the following hold:
  \begin{enumerate}
  \item \label{i:p:A-lattice} The partial order $(\Fin(M) / {\sim}, \leq)$ is a lattice.
  \item \label{i:p:A-wf} The relation $<$ on $\Fin(M)$ is well-founded.
  \end{enumerate}
\end{theorem}
\begin{proof}
  \autoref{i:p:A-lattice} Let $A, B \in \Fin(M)$. The least upper bound $[A] \vee [B]$ is given by $[A \cup B]$. That $[A \cup B]$ is an upper bound follows from \autoref{eq:sub-implies-leq}. If $[C] \geq [A],[B]$, then $A \cup B \sub \acl C$, so $\acl(A \cup B) \sub \acl C$ by \autoref{l:acl}.

  For the greatest lower bound, let $D = \acl A \cap \acl B$ and note that $D$ is algebraically closed and locally finite. It is also definable because $G_{(A \cup B)}$ is a subgroup of $G_{D}$. Use \autoref{l:stab-loc-fin} to find a finite $D_0 \sub D$ with $D = \acl D_0$ and $G_D = G_{D_0}$. We show that $[D_0] = [A] \wedge [B]$. It is clear that $[D_0] \leq [A], [B]$. If $C \in \Fin(M)$ is such that $C \leq A, B$, then $C \sub \acl A \cap \acl B$, so $\acl C \sub \acl A \cap \acl B = \acl D_0$, and it follows that $C \leq D_0$.

  \autoref{i:p:A-wf} Suppose, towards a contradiction, that there exist $A_0, A_1, \ldots \in \Fin(M)$ with $A_0 > A_1 > \cdots$. Let $D_i = \acl A_i$ and note that $D_i \supseteq D_{i+1}$ and $A_i \sub D_j$ for all $i \geq j$. Let $H_i = G_{(D_i)}$, so that $H_0 \leq H_1 \leq \cdots$.

  We prove that for every $i$, the inclusion $H_i \cdot A_0 \sub H_{i+1} \cdot A_0$ is proper. By assumption, $G_{A_{i+1}} \cdot A_i$ is infinite, so $H_{i+1} \cdot A_i$ is also infinite by \autoref{l:pointwise-st} \autoref{i:ps:1}. On the other hand, $A_i \sub \acl A_0$, so $G_{A_0} \cdot A_i$ is finite. Hence there exists $g \in H_{i+1}$ with $g^{-1} \cdot A_i \notin G_{A_0} \cdot A_i$. We will show that $g \cdot A_0 \notin H_i \cdot A_0$. Suppose, to the contrary, that there is $h \in H_i$ with $h \cdot A_0 = g \cdot A_0$. Then $g^{-1}h \in G_{A_0}$ and $g^{-1}h \cdot A_i = g^{-1} \cdot A_i$, contradicting the fact that $g^{-1} \cdot A_i \notin G_{A_0} \cdot A_i$.

  We obtained that $H_0 \cdot A_0 \subsetneq H_1 \cdot A_0 \subsetneq \cdots$ is an infinite strictly increasing sequence of $H_0$-invariant subsets of $G \cdot A_0$, which contradicts \autoref{l:pointwise-st} \autoref{i:ps:2}.
\end{proof}

\subsection{The universal action and unitary representations}
\label{sec:unit-repr}

Let $G$ be a Roelcke precompact, non-archimedean, Polish group. Among all actions of $G$ on a countable set, there is a universal one that can be constructed as follows. Let $(V_i : i \in I)$ be a collection of representatives for the equivalence relation of conjugacy on open subgroups of $G$. Note that the set $I$ is countable by \cite{Tsankov2012}*{Corollary~2.5}. Let
\begin{equation*}
  \cU = \bigsqcup_i G/V_i
\end{equation*}
and equip it with the left translation action. 
Note that for any open $V \leq G$, there exists $a \in \cU$ with $V = G_a$. This implies, in particular, that the action $G \actson \cU$ is topologically faithful (see \cite{Evans2016}*{Lemma~1.9}).
This action is universal in the sense that every continuous, transitive action of $G$ on a countable set embeds into it and it is minimal with this property. If one wants to embed any continuous action of $G$ on a countable set, one has to take infinitely many copies of $\cU$.

If $G$ is the automorphism group of an $\aleph_0$-categorical structure $M$, then the set $\cU$ can be constructed model-theoretically essentially as the structure $M\eq$. 

The universal action has canonical representatives for the $\sim$-equivalence classes discussed in the previous subsection. Define $s \colon \Fin(\cU) \to \cU$ by
\begin{equation*}
  s(A) = \text{ the unique $a \in \cU$ such that } G_a = G_{\acl A}.
\end{equation*}
Indeed, by \autoref{l:acl} and \autoref{l:stab-loc-fin}, $G_{\acl A}$ is an open subgroup of $G$. Let $V_i$ be the representative of its conjugacy class that appears in the definition of $\cU$ and let $g \in G$ be such that $gV_ig^{-1} = G_{\acl A}$. Then one can take $a = g V_i$. To verify uniqueness, note that $G_{gV_i} = g V_i g^{-1}$. By the construction of $\cU$, if $G_{gV_i} = G_{hV_j}$, then $V_i = V_j$ and $h^{-1} g \in N_G(V_i)$. By \autoref{l:stab-loc-fin}, $V_i$ is self-normalizing, so $hV_i = gV_i$.

It is easy to check that $A \sim s(A)$ for all $A \in \Fin(\cU)$ and that $s$ is idempotent in the sense that $s(\set{s(A)}) = s(A)$ for all $A \in \Fin(M)$. Thus the set
\begin{equation*}
  \cA \coloneqq \set{s(A) : A \in \Fin(\cU)},
\end{equation*}
which we also identify with a subset of $\Fin(\cU)$ by confounding $a$ with $\set{a}$, is a complete section for $\sim$ and $(\cA, \leq) \cong (\Fin(\cU) / {\sim}, \leq)$.

We recall that a \df{unitary representation} of $G$ is a continuous action on a complex Hilbert space $\cH$ by unitary isomorphisms.
If $G \actson M$ is any continuous action on a countable set, we have a natural representation $G \actson \ell^2(M)$. The representation $G \actson \ell^2(\cU)$ is universal in the following sense.
\begin{fact}[\cite{Tsankov2012}*{Theorem~4.2}, \cite{Jahel2022}*{Fact~3.1}]
  \label{f:urep}
  Let $G$ be a Roelcke precompact, non-archimedean, Polish group, and let $G \actson \cU$ be its universal action. Then every unitary representation of $G$ is isomorphic to a subrepresentation of a direct sum of copies of $\ell^2(\cU)$.
\end{fact}

Let now $G \actson \cH$ be a unitary representation of $G$. For $a \in \cA$, we denote
\begin{equation*}
  \cH_a = \cl{\set{\xi \in \cH : G_a \cdot \xi \text{ is finite}}}.
\end{equation*}
It is clear that $\cH_a$ is a closed subspace of $\cH$ and that $\cH_a \sub \cH_b$ for $a \leq b$.

If $\cH_1, \cH_2, \cH_3$ are closed subspaces of a Hilbert space $\cH$ with $\cH_2 \sub \cH_1 \cap \cH_3$, we write
\begin{equation*}
  \cH_1 \Perp[\cH_2] \cH_3
\end{equation*}
if the orthogonal complements of $\cH_2$ in $\cH_1$ and in $\cH_3$ are orthogonal. If $p_1, p_2, p_3$ denote the orthogonal projections on $\cH_1, \cH_2, \cH_3$, respectively, this is equivalent to $p_3p_1 = p_2p_1$.

The next proposition is similar to \cite{Jahel2022}*{Proposition~3.2}; see also \cite{BenYaacov2018}. 
\begin{prop}
  \label{p:unitary-rep}
  Let $G$ be a non-archimedean, Roelcke precompact, Polish group and let $G \actson \cU$ be the universal action of $G$. Let $G \actson^\pi \cH$ be a unitary representation. Then, for all $a, b \in \cA$,
  \begin{equation}
    \label{eq:p:ur}
    \cH_a \Perp[\cH_{a \wedge b}] \cH_b.
  \end{equation}
\end{prop}
\begin{proof}
  First we check the condition for $\cH = \ell^2(\cU)$. The main observation is that
  \begin{equation*}
    \cH_a = \set{f \in \ell^2(\cU) : \supp f \sub \acl a}.
  \end{equation*}
  Notice that the subspace on the right-hand side is closed.
  To see the $\sub$ inclusion, let $f \in \cH_a$ be such that $G_a \cdot f$ is finite. Then there is an open finite-index subgroup $V \leq G_a$ which fixes $f$. If we suppose that there is $b \in M \sminus \acl a$ such that $f(b) \neq 0$, we have that $f$ is constant on $V \cdot b$, which is infinite, and this contradicts the fact that $f \in \ell^2(\cU)$. Conversely, write $\acl a = \bigcup_i A_i$, where each $A_i$ is finite, $G_a$-invariant, and $A_i \sub A_{i+1}$. Then for each $f$ with $\supp f \sub \acl a$, we have that $f \chi_{A_i} \to f$ and $G_a \cdot (f \chi_{A_i})$ is finite for each $i$.
  Now, recalling that $\acl(a \wedge b) = \acl a \cap \acl b$, the conclusion is clear for $\cH = \ell^2(\cU)$.
  
  From \autoref{f:urep}, we have that $\pi$ is a subrepresentation of a direct sum of copies of $\ell^2(\cU)$.
  If $\cK$ and $\cH$ are two representations and $a \in \cA$, then $(\cH \oplus \cK)_a = \cH_a \oplus \cK_a$, so it is clear that \eqref{eq:p:ur} passes to direct sums. To see that it passes to subrepresentations, denote by $P^\cH_a$ the projection onto $\cH_a$ and observe that \eqref{eq:p:ur} can be equivalently written as $P^\cH_b P^\cH_a = P^\cH_{a \wedge b} P^{\cH}_a$. Let now $\cK \sub \cH$ be a subrepresentation and suppose that \eqref{eq:p:ur} holds for $\cH$. Then $\cH_a = \cK_a \oplus \cK_a^\perp$, and similarly for $b$ and $a \wedge b$. Then we have, for any $\xi \in \cK$,
  \begin{equation*}
    P^\cK_b P^\cK_a \xi = P^\cH_b P^\cH_a \xi = P^\cH_{a \wedge b} P^\cH_a \xi = P^\cK_{a \wedge b} P^\cK_a \xi. \qedhere
  \end{equation*}
\end{proof}

We conclude this section with a variant of a well-known lemma. First, we have the following (see, e.g., \cite{Tsankov2012}*{Lemma~3.1}).
\begin{lemma}
  \label{l:dense-fixed-basis}
  Let $G$ be a non-archimedean group and let $G \actson \cH$ be a unitary representation. Let $(V_i : i \in I)$ be a basis of open subgroups at $1_G$. Then the set
  \begin{equation*}
    \set{\xi \in \cH : \exists i \in I \ V_i \cdot \xi = \xi}
  \end{equation*}
  is dense in $\cH$.
\end{lemma}
\begin{cor}
  \label{c:Ha-dense}
  Let $G$ be a Roelcke precompact, non-archimedean, Polish group, and let $G \actson \cH$ be a unitary representation. Let the set $\cA$ and the subspaces $\cH_a$ for $a \in \cA$ be defined as above. Then $\bigcup_{a \in \cA} \cH_a$ is dense in $\cH$.
\end{cor}
\begin{proof}
  Let $\xi \in \cH$ be arbitrary and let $\eps > 0$. Recalling that every open subgroup of $G$ is the stabilizer of some point in $\cU$, by \autoref{l:dense-fixed-basis}, there exists $\xi_0 \in \cH$ and $b \in \cU$ such that $G_b \cdot \xi_0 = \xi_0$ and $\nm{\xi - \xi_0} < \eps$. Let $a \in \cA$ be such that $\acl a = \acl b$. Then $G_a \cap G_b$ has finite index in $G_b$, so $G_a \cdot \xi_0$ is finite and $\xi_0 \in \cH_a$.
\end{proof}


\section{Non-singular actions of Roelcke precompact, non-archimedean groups}
\label{sec:non-singular-actions-1}

\subsection{Non-singular actions}
\label{sec:non-singular-actions}

Let $(\Omega, \cB, \mu)$ be a $\sigma$-finite measure space and let $\cN = \set{B \in \cB : \mu(B) = 0}$ be its null ideal. The \df{measure algebra} $\MALG(\mu)$ is the quotient Boolean algebra $\cB / \cN$. We denote by $\Aut^*(\mu)$ the group of automorphisms of the Boolean algebra $\MALG(\mu)$. Another way to view $\Aut^*(\mu)$ is as the group of all bi-measurable bijections $g$ of $\Omega$ which preserve the null ideal (or, equivalently, such that $g_*\mu$ is equivalent to $\mu$), where two such bijections are identified if they are equal on a co-null set. The $*$ in the notation $\Aut^*(\mu)$ is there to distinguish it from its subgroup $\Aut(\mu)$ consisting of the transformations which also preserve the measure $\mu$. Note that $\Aut^*(\mu)$ only depends on the measure class of $\mu$; in particular, we can always replace $\mu$ by a probability measure.

$\Aut^*(\mu)$ embeds into the unitary group $U(L^2(\Omega))$ via its \df{Koopman representation}:
\begin{equation}
  \label{eq:koopman}
  (g \cdot f)(x) = \Big( \frac{\ud g_* \mu}{\ud \mu} \Big)^{1/2}(x) f(g^{-1} \cdot x), \quad \text{ for } g \in \Aut^*(\mu), f \in L^2(\Omega),
\end{equation}
where $\frac{\ud g_* \mu}{\ud \mu}$ denotes the Radon--Nikodym derivative. The Hilbert space $L^2(\Omega)$ is also endowed with the structure of a (complex) Banach lattice, where the positive elements are the positive functions and the elements of the image of the embedding of $\Aut^*(\mu)$ are precisely the elements of $U(L^2(\Omega))$ preserving the order. Thus the image of $\Aut^*(\mu)$ in $U(L^2(\Omega))$ is closed and we equip $\Aut^*(\mu)$ with the group topology coming from this embedding.

If $f \in L^2(\Omega)$, we denote by $\supp f \in \MALG(\mu)$ the \df{support} of $f$, which is the set of $x \in \Omega$ such that $f(x) \neq 0$. Note that if $f_1, f_2 \in L^2(\Omega)$ are positive, then
\begin{equation*}
  f_1 \perp f_2 \iff \supp f_1 \cap \supp f_2 = \emptyset.
\end{equation*}

\begin{defn}
  \label{df:non-singular}
  Let $G$ be a topological group. A \df{non-singular action} of $G$ on the $\sigma$-finite measure space $(\Omega, \cB, \mu)$ is a continuous homomorphism $G \to \Aut^*(\mu)$. The action is called \df{ergodic} if it has no fixed points in $\MALG(\mu)$ apart from $\bZero$ and $\bOne$. The action \df{preserves} the measure $\mu$ (as opposed to just the measure class) if the image of the homomorphism is in the subgroup $\Aut(\mu)$, i.e., $\mu(g \cdot E) = \mu(E)$ for all $E \in \MALG(\mu)$.
\end{defn}

\begin{remark}
  \label{rem:boolean-actions}
  An action on $\MALG(\mu)$ as above is often called in the literature a \df{Boolean action}. It is clear that any pointwise Borel action of $G$ on $\Omega$ such that $g_*\mu \sim \mu$ for all $g \in G$ gives rise to a Boolean action. The converse, for pmp actions and non-archimedean Polish groups, is also true, as shown by Glasner and Weiss in \cite{Glasner2005a}*{Theorem~2.3}. \autoref{th:non-singular-nA-RP} below implies that the same holds for non-singular actions of Roelcke precompact, non-archimedean, Polish groups.
\end{remark}

The following lemma gives a simple criterion for a non-singular action to actually preserve a measure.
\begin{lemma}
  \label{l:fixed-vector}
  Let $G \actson (\Omega, \cB, \mu)$ be a non-singular action. Then the following are equivalent:
  \begin{enumerate}
  \item The action $G \actson L^2(\Omega)$ has a non-zero fixed vector;
  \item There is a probability measure $\nu \prec \mu$, which is $G$-invariant.
  \end{enumerate}
\end{lemma}
\begin{proof}
  \begin{cycprf}
  \item[\impnext] Suppose that $f_0 \in L^2(\Omega)$ with $\nm{f_0} = 1$ is fixed. Then $|f_0|$ is also fixed and positive of norm $1$. Let $\nu$ be the probability measure defined by $\RaNi{\nu}{\mu} = |f_0|^2$. Then, for every $g \in G$, using the equation \autoref{eq:koopman}, we have
    \begin{equation*}
      \Big( \RaNi{\nu}{\mu} \Big)^{1/2} = \Big( \RaNi{g_* \mu}{\mu} \Big)^{1/2} \Big( \RaNi{g_* \nu}{g_*\mu} \Big)^{1/2} = \Big( \RaNi{g_* \nu}{\mu} \Big)^{1/2},
    \end{equation*}
    showing that $\nu$ is invariant.
    
  \item[\impfirst]
    We will check that $\big(\RaNi{\nu}{\mu}\big)^{1/2} \in L^2(\Omega)$ is invariant. Indeed,
    \begin{equation*}
      \begin{split}
        \Big( g \cdot \Big(\RaNi{\nu}{\mu}\Big)^{1/2} \Big)(x)
          &= \Big(\RaNi{g_* \mu}{\mu}\Big)^{1/2}(x) \Big(\RaNi{\nu}{\mu}\Big)^{1/2}(g ^{-1} \cdot x) \\
          &= \Big(\RaNi{g_* \mu}{\mu}\Big)^{1/2}(x) \Big(\RaNi{g_* \nu}{g_* \mu}\Big)^{1/2}(x) \\
          &= \Big(\RaNi{g_* \mu}{\mu}\Big)^{1/2}(x) \Big(\RaNi{\nu}{g_* \mu}\Big)^{1/2}(x) = \Big(\RaNi{\nu}{\mu}\Big)^{1/2}(x).
            \qedhere
      \end{split}
    \end{equation*}
  \end{cycprf}
\end{proof}

If $G \actson^{\alpha_i} (\Omega_i, \cB_i, \mu_i)$, $i \in I$ is a countable family of non-singular actions, one can form the \df{disjoint union action} on $\bigsqcup_i \Omega_i$ with measure $\sum_i \mu_i$ defined by:
\begin{equation*}
  g \cdot \bigsqcup_i A_i = \bigsqcup_i g \cdot A_i.
\end{equation*}
Note that this action is never ergodic unless possibly if $|I| = 1$.

\subsection{Induced actions}
\label{sec:induced-actions}

In this subsection, we recall the standard construction of induction.
Let $G$ be a topological group and let $V \leq G$ be an open subgroup of countable index. We denote by $\xi$ the counting measure on $G/V$.
If $V \actson^\sigma (\Omega, \cB, \mu)$ is a non-singular action, we can construct an action of $G$ on $(G/V \times \Omega, \cP(G/V) \otimes \cB, \xi \otimes \mu)$ as follows. Let $s \colon G/V \to G$ be a section for the quotient map $G \to G/V$, i.e., a map such that $s(gV) \in gV$ for all $g \in G$. Define the cocycle $c \colon G \times G/V \to V$ by
\begin{equation}
  \label{eq:cocycle-def}
  c(g, hV) = s(ghV)^{-1}g s(hV)
\end{equation}
and note that it satisfies the \df{cocycle identity}:
\begin{equation*}
  c(g_1 g_2, hV) = c(g_1, g_2hV)c(g_2, hV) \quad \text{ for } g_1, g_2, h \in G.
\end{equation*}
Then one can define the \df{induced action} $\Ind_V^G(\sigma)$ of $G$ on $\MALG(\xi \otimes \mu)$ by
\begin{equation*}
  g \cdot (\set{hV} \times A) = \set{ghV} \times \sigma(c(g, hV)) \cdot A \quad \text{ for } g, h \in G, A \in \MALG(\mu).
\end{equation*}
The action does not depend on the choice of the section $s$ (up to conjugacy).
Note that if $\sigma$ preserves the measure $\mu$, then $\Ind_V^G(\sigma)$ preserves the measure $\xi \otimes \mu$. Also, if $\sigma$ is ergodic, then so is $\Ind_V^G(\sigma)$.

\subsection{A classification theorem for non-singular actions}
\label{sec:class-theor-non}

\begin{theorem}
  \label{th:non-singular-nA-RP}
  Let $G$ be a Roelcke precompact, non-archimedean, Polish group and let $G \actson (\Omega, \cB, \mu)$ be a non-singular action. Then there exist countably many self-commensurating, open subgroups $(V_i : i \in I)$ of $G$ and measure-preserving actions $V_i \actson^{\alpha_i} (\Omega_i, \cB_i, \mu_i)$, with each $\mu_i$ a probability measure, such that the original action $G \actson \Omega$ is isomorphic to the disjoint union
  \begin{equation}
    \label{eq:th:non-singular-nA-RP}
    \bigsqcup_i \Ind_{V_i}^G(\alpha_i).
  \end{equation}
  In particular, every non-singular action of $G$ is isomorphic to a measure-preserving action (with a possibly infinite measure) and every ergodic such action is of the form $\Ind_V^G(\alpha)$, where $V$ and $\alpha$ are as above.
\end{theorem}
\begin{proof}
  First, we show that there is a non-null $G$-invariant set $E \in \MALG(\mu)$ such that the action of $G$ on $E$ is isomorphic to $\Ind_V^G(\alpha)$ as above. Let $\cH = L^2(\Omega, \mu)$ and consider the Koopman representation $\kappa$ of $G$ on $\cH$ given by \eqref{eq:koopman}. Let $(\cA, \leq)$ and the subspaces $\cH_a$ for $a \in \cA$ be defined as in Subsection~\ref{sec:unit-repr}. By \autoref{c:Ha-dense}, $\bigcup_{a \in \cA} \cH_a$ is dense in $\cH$, so there exists $a$ with $\cH_a \neq 0$. By \autoref{th:A-partial-order}, there is a $<$-minimal element $a \in \cA$ with $\cH_a \neq 0$. 

  Let $f_0 \in \cH_{a}$ be non-zero with $G_{a} \cdot f_0 = \set{f_0, \ldots, f_n}$ finite. Then $f_a \coloneqq \sum_j |f_j|$ is a $G_{a}$-invariant positive element of $\cH_a$ and, after normalizing, it gives rise to a $G_{a}$-invariant probability measure $\nu_a \prec \mu$ as in \autoref{l:fixed-vector}. Let $E = \supp \nu_a$ and note that $E$ is $G_a$-invariant. Let $T$ be a set of representatives of the cosets in $G/G_{a}$ with $1_G \in T$. We will show that for $t_1 \neq t_2 \in T$, $t_1 \cdot E \cap t_2 \cdot E = \emptyset$. It suffices to see that for $t \neq 1_G$, $t \cdot E \cap E = \emptyset$. By \autoref{p:unitary-rep}, $\cH_a \Perp[\cH_{a \wedge t \cdot a}] \cH_{t \cdot a}$. We also have that $t^{-1} \cdot a \notin \acl a$ by \autoref{l:stab-loc-fin}, so $a \wedge t \cdot a < a$ and by the choice of $a$, $\cH_{a \wedge t \cdot a} = 0$. Now $\big(\RaNi{\nu_a}{\mu}\big)^{1/2} = f_a \in \cH_a$ and $\big(\RaNi{(t \cdot \nu_a)}{\mu}\big)^{1/2} = t \cdot f_a \in \cH_{t \cdot a}$. This shows that the two Radon--Nikodym derivatives are orthogonal, so the measures $\nu_a$ and $t \cdot \nu_a$ have disjoint supports, i.e., $E \cap t \cdot E = \emptyset$, as desired.

  Denote the probability measure-preserving action of $G_a$ on $(E, \nu_a)$ by $\alpha$. We will see that the actions
  $\Ind_{G_a}^G(\alpha)$ and $G \actson \bigsqcup_{t \in T} t \cdot E$ are isomorphic. In order to make the isomorphism more transparent, we choose the section $s \colon G/G_a \to G$ appearing in the definition of the cocycle \autoref{eq:cocycle-def} to have image $T$, i.e,
  \begin{equation*}
    s(gG_a) = t \iff t \in T \And tG_a = gG_a.
  \end{equation*}
  Now define a map
  \begin{equation*}
    \Phi \colon \MALG(\Ind_{G_a}^G(\alpha)) \to \MALG(\bigsqcup_{t \in T} t \cdot E, \sum_t t \cdot \nu_a)
  \end{equation*}
  by
  \begin{equation*}
    \Phi(\set{tG_a} \times D) = t \cdot D \quad \text{for } t \in T, D \in \MALG(E, \nu_a).
  \end{equation*}
  It is straightforward to check that $\Phi$ is an isomorphism.

  To conclude with the proof of the theorem, we apply Zorn's lemma. Let $\cD$ be the collection of all elements of $\MALG(\mu)$ which are $G$-invariant and such that the restriction of the action of $G$ to them is of the form \eqref{eq:th:non-singular-nA-RP}, and order it by inclusion. If $\cC \sub \cD$ is a chain, we can always find a cofinal set $\cC' \sub \cC$, which is countable, because $(\cC, \sub)$ is isomorphic to $(\set{\mu(C) : C \in \cC}, \leq) \sub (\R, \leq)$. It is clear that $\bigcup \cC'$ is an upper bound for $\cC$. Let $D$ be a maximal element of $\cD$. If $D \neq \Omega$, we can apply what we already proved to the action of $G$ on $\Omega \sminus D$ and contradict its maximality.
\end{proof}


\section{A generalization of de Finetti's theorem}
\label{sec:some-gener-deFinetti}

In this section, we prove a classification result for some specific pmp actions of permutation groups, inspired by de Finetti's theorem. We let $G \leq \Sym(M)$ be a permutation group with $M$ countable. We equip $G$ with the topology inherited from $\Sym(M)$, so that $G$ becomes a non-archimedean group, and consider pmp actions of $G$ continuous for this topology. (We call a pmp action $G \actson (Z, \mu)$ \df{continuous} if the corresponding morphism $G \to \Aut(\mu)$ is continuous.) One loses nothing if one requires in addition that $G$ is a closed subgroup of $\Sym(M)$: every continuous action of $G$ extends to its closure in $\Sym(M)$. Moreover if $G$ is closed (i.e., Polish), then every Borel pmp action $G \actson (Z, \mu)$ on a standard probability space gives rise to a continuous homomorphism $G \to \Aut(\mu)$. Our main goal is to classify, under suitable assumptions on $G$, the $G$-invariant Borel probability measures on the space $\Omega \coloneqq [0, 1]^M$ under the shift action
\begin{equation*}
  (g \cdot \omega)(a) = \omega(g^{-1} \cdot a), \quad \omega \in \Omega, a \in M.
\end{equation*}
It will be helpful to adopt a probabilistic viewpoint. For each $a \in M$, we consider the projection $\xi_a \colon \Omega \to [0, 1]$, $\xi_a(\omega) = \omega(a)$ as a random variable and the problem becomes to classify all possible $G$-invariant distributions $\mu$ of the variables $(\xi_a : a \in M)$. 
A $\sigma$-field of events corresponds simply to a closed subalgebra of $\MALG(\mu)$. If $\cG \sub \MALG(\mu)$ is a $\sigma$-field, we denote by $\E_\cG$ the conditional expectation relative to $\cG$. We let $L^2(\cG)$ be the closed subspace of $L^2(\Omega)$ consisting of $\cG$-measurable functions and recall that $\E_\cG$ restricted to $L^2(\Omega)$ is the orthogonal projection on $L^2(\cG)$. If $\cA$ is a collection of events or random variables, we denote by $\gen{\cA}$ the $\sigma$-field generated by $\cA$.

The tail $\sigma$-field is a classical object in the study of random processes; here we will define a dynamical variant of it that exists in any pmp $G$-system for a permutation group $G$.
\begin{defn}
  \label{df:sigma-fields}
  Let $G \leq \Sym(M)$ be a permutation group and let $G \actson (Z, \mu)$ be a continuous pmp action.
  For a finite $A \sub M$, we denote
  \begin{equation*}
    \cF_A \coloneqq \set{E \in \MALG(\mu) : g \cdot E = E \text{ for all } g \in G_{(A)}}.
  \end{equation*}
  If $a \in M$, we write $\cF_a$ instead of $\cF_{\set{a}}$. If $D \sub M$ is an arbitrary subset, we let
  \begin{equation*}
    \cF_D \coloneqq \cl{\bigcup \set[\big]{\cF_A : A \sub D \text{ finite}}}.
  \end{equation*}
  We define the \df{invariant $\sigma$-field} as
  \begin{equation*}
    \cJ \coloneqq \cF_{\emptyset} = \set{E \in \MALG(\mu) : g \cdot E = E \text{ for all } g \in G}.
  \end{equation*}
  The action $G \actson (Z, \mu)$ is \df{ergodic} if $\cJ$ is trivial.
  Finally, we define the
  \df{dynamical tail $\sigma$-field} as
  \begin{equation*}
    \cT \coloneqq \bigcap \set{\cF_{M \sminus A} : A \sub M \text{ finite}}.
  \end{equation*}
\end{defn}
We note that $\cF_C \sub \cF_D$ for all $C, D \sub M$, $g \cdot \cF_D = \cF_{g \cdot D}$, and $g \cdot \cT = \cT$ for all $g \in G$. It follows from \autoref{l:dense-fixed-basis} that
\begin{equation}
  \label{eq:FM-everything}
  \cF_M = \MALG(\mu).
\end{equation}

The following permutation group lemma is well-known. See, for example, \cite{Hodges1993}*{Lemma~4.2.1} for a proof.
\begin{lemma}[Neumann]
  \label{l:Neumann}
  Let $G \leq \Sym(M)$ be a permutation group. Then the following are equivalent:
  \begin{itemize}
  \item The orbits of $G$ are infinite; 
  \item For all finite $A, B \sub M$ there exists $g \in G$ with $g \cdot A \cap B = \emptyset$.
  \end{itemize}
\end{lemma}

\begin{prop}
  \label{p:inv-sub-tail}
  Suppose that the action $G \actson M$ has infinite orbits. Then $\cJ \sub \cT$.
\end{prop}
\begin{proof}
  Let $E \in \cJ$. It follows from \autoref{eq:FM-everything} that there exists a finite set $A \sub M$ and a $\cF_A$-measurable event $E'$ with $\mu(E \sdiff E') < \eps$. If $B \sub M$ is any finite set, by \autoref{l:Neumann}, there is $g \in G$ with $g \cdot A \cap B = \emptyset$. By invariance of $E$, this implies that for any $B$, there is an $\cF_{M \sminus B}$-measurable event $E''$ (namely, $E'' = g \cdot E'$) with $\mu(E \sdiff E'') < \eps$. Taking limits as $B$ exhausts $M$ and $\eps \to 0$, we obtain that $E \in \cT$.
\end{proof}

If $\cG_1, \cG_2, \cG_3$ are three $\sigma$-fields, we write $\cG_1 \indep[\cG_2] \cG_3$ to denote that $\cG_1$ and $\cG_3$ are conditionally independent over $\cG_2$. We refer to \cite{Kallenberg2002} for the basic properties of conditional independence.
\begin{prop}
  \label{p:tail-indep}
  Let $G \leq \Sym(M)$ be a permutation group and let $A \sub G$ be such that the orbits of the action $G_{(A)} \actson M \sminus A$ are infinite. Let $G \actson (Z, \mu)$ be any continuous pmp action. Then $\cF_A \indep[\cT] \cF_{M \sminus A}$.
\end{prop}
\begin{proof}
  Let $M = \bigcup_n S_n$ with $S_n$ finite and increasing and let $\cG_n = \cF_{M \sminus S_n}$. Then $\cG_0 \supseteq \cG_1 \supseteq \cdots$ and $\bigcap_n \cG_n = \cT$. 
  Let $B \sub M \sminus A$ be finite.
  Using \autoref{l:Neumann}, for each $n$, find $g_n \in G_{(A)}$ such that $g_n \cdot B \cap S_n = \emptyset$, so that, in particular, $\cF_{g_n \cdot B} \sub \cG_n$. Let now $\xi$ be any bounded $\cF_A$-measurable random variable. Note that $g \cdot \xi = \xi$ for all $g \in G_{(A)}$. 

Let $\eps > 0$. By reverse martingale convergence \cite{Durrett2005}*{Theorem~5.6.3}, we have $\E_{\cG_n} \xi \to \E_\cT \xi$ in $L^2$; let $n$ be such that $\nm{\E_{\cG_n} \xi} - \nm{\E_{\cT} \xi} < \eps$ (here $\nm{\cdot}$ is the $L^2$-norm). We have:
\[
\begin{split}
  \nm{\E_{\cT \vee \cF_B} \xi} &= \nm{\E_{g_n \cdot \cT \vee \cF_{g_n \cdot B}} g_n \cdot \xi} \\
  &=\nm{\E_{\cT \vee \cF_{g_n \cdot B}} \xi} \\
  &\leq \nm{\E_{\cG_n \vee \cF_{g_n \cdot B}} \xi} \\
  &= \nm{\E_{\cG_n} \xi} \leq \nm{\E_\cT \xi} + \eps.
\end{split}
\]
As this is true for all $\eps$, we obtain that $\nm{\E_{\cT \vee \cF_B} \xi} = \nm{\E_\cT \xi}$, whence $\E_{\cT \vee \cF_B} \xi = \E_\cT \xi$, showing that $\cF_A$ and $\cF_B$ are independent over $\cT$ \cite{Kallenberg2002}*{Proposition~5.6}.
\end{proof}

\begin{cor}
  \label{c:tail-single}
  Let $G \leq \Sym(M)$ be a permutation group such that for every $a \in M$, the orbits of $G_a$ on $M \sminus \set{a}$ are infinite. Let $G \actson (Z, \mu)$ be a continuous pmp action. Then the $\sigma$-fields $\cF_a$ are conditionally independent over $\cT$.
\end{cor}
\begin{proof}
    To prove the conclusion, it suffices to check that for every $a \in M$ and finite $B \sub M$ with $a \notin B$, $\cF_a \indep[\cT] \cF_B$. This follows from \autoref{p:tail-indep} applied with $A = \set{a}$.
\end{proof}

\begin{prop}
  \label{p:gen-tail}
  Let $G \leq \Sym(M)$ and $G \actson (Z, \mu)$ satisfy the assumptions of \autoref{c:tail-single}. Then:
  \[
    \cT \cap \bigvee_{a \in M} \cF_a = \bigvee_{a \in M} \cT \cap \cF_a.
  \]
\end{prop}
\begin{proof}
  Let $\cG = \bigvee_a \cF_a$. 
  Let $a_1, \ldots, a_n$ be distinct elements of $M$ and let, for each $i$, $\xi_i$ be a bounded $\cF_{a_i}$-measurable random variable. Then, by \autoref{c:tail-single},
  \begin{equation}
    \label{eq:indep-tail}
    \E_{\cT} \xi_{1} \cdots \xi_{n} = (\E_{\cT} \xi_{1}) \cdots (\E_{\cT} \xi_{n}).
  \end{equation}
  For every $g \in G_{a_i}$, we have that $g \cdot \E_\cT \xi_i = \E_{g \cdot \cT} g \cdot \xi_i = \E_\cT \xi_i$, so $\E_\cT \xi_i \in \cF_{a_i}$. In particular, $\E_\cT \xi_1 \cdots \xi_n$ is $\cG$-measurable, so also $\cT \cap \cG$-measurable. As $\E_{\cT \cap \cG} \E_\cT = \E_{\cT \cap \cG}$, this implies that $\E_\cT \xi_1 \cdots \xi_n = \E_{\cT \cap \cG} \xi_1 \cdots \xi_n$. Linear combinations of variables of the form $\xi_1 \cdots \xi_n$ generate a dense subspace of $L^2(\cG)$, so $\E_\cT \xi_1 \cdots \xi_n$ generate a dense subspace of $L^2(\cT \cap \cG)$. However, by \autoref{eq:indep-tail}, each of them belongs to $\bigvee_a \cT \cap \cF_a$ and we are done.
\end{proof}

For the rest of the section, we will specialize to the action $G \actson ([0, 1]^M, \mu)$, where $\mu$ is an arbitrary invariant probability measure. We recall that $\xi_a$ denotes the projection on the coordinate $a$. In particular, $\MALG(\mu) = \gen{\xi_a : a \in M} = \bigvee_a \cF_a$. In this setting, we also have the usual tail $\sigma$-field.
\begin{defn}
  \label{df:usual-tail}
  Let $(\xi_a : a \in M)$ be a collection of random variables. Then the \df{tail $\sigma$-field} is
  \begin{equation*}
    \cT_0 = \bigcap \set[\big]{\gen{\xi_a : a \in M \sminus A} : A \sub M \text{ finite}}.
  \end{equation*}
\end{defn}


\begin{remark}
  \label{rem:tail0-instead-of-tail}
  We note that \autoref{p:tail-indep} and \autoref{c:tail-single} remain true if one replaces $\cT$ with $\cT_0$. In the proofs, one simply needs to replace the $\sigma$-fields $\cF_A$ (for $A \sub M$) by $\cF_A'$ defined by
  \begin{equation*}
    \cF_A' = \gen{\xi_a : a \in A}.
  \end{equation*}
\end{remark}

In fact, under a somewhat stronger hypothesis on the action $G \actson M$, the two tail $\sigma$-fields coincide.
We recall that $G \actson M$ \df{has no algebraicity} if for every finite $A \sub M$, the action $G_{(A)} \actson M \sminus A$ has infinite orbits.

\begin{prop}
  \label{p:t-equal-t0}
  Suppose that the action $G \actson M$ has no algebraicity. Then $\cT = \cT_0$.
\end{prop}
\begin{proof}
  As $\xi_a$ is $\cF_a$-measurable for every $a \in M$, the inclusion $\cT_0 \sub \cT$ is clear. For the converse, let $E \in \cT$, let $C \sub M$ be finite, and let $\eps > 0$. There exists a finite $A \sub M$ and $E' \in \gen{\xi_a : a \in A}$ with $\mu(E \sdiff E') < \eps$. As $E \in \cT$, there also exists a finite $B \sub M$ with $B \cap A = \emptyset$ and $E'' \in \cF_B$ such that $\mu(E \sdiff E'') < \eps$. Using the no algebraicity assumption and \autoref{l:Neumann}, let $g \in G_{(B)}$ be such that $g \cdot A \cap C = \emptyset$. Then $g \cdot E' \in \gen{\xi_a : a \in M \sminus C}$ and
  \begin{equation*}
    \begin{split}
      \mu(E \sdiff g \cdot E') &\leq \mu(E \sdiff E'') + \mu(E'' \sdiff g \cdot E') \\
                               &\leq \eps + \mu(E'' \sdiff E') \leq 3 \eps.
    \end{split}
  \end{equation*}
  As $C$ and $\eps$ were arbitrary, this shows that $E \in \cT_0$.
\end{proof}

In order to motivate the next definition, we recall some terminology and facts from model theory. For a permutation group $G \leq \Sym(M)$, we will say that two finite tuples $\bar a, \bar b \in M^n$ \df{have the same type} (notation: $\tp \bar a = \tp \bar b$) if they are in the same $G$-orbit for the diagonal action $G \actson M^n$. Two countable tuples $\bar a, \bar b$ \df{have the same type} if for all $n \in \N$, $\tp (\bar a|_n) = \tp (\bar b|_n)$. It is easy to see that having no algebraicity is equivalent to the following condition: for every $n \in \N$, every $\bar b \in M^n$, and $a \in M$ not from $\bar b$, there exists $a' \neq a$ such that $\tp (\bar b a) = \tp(\bar b a')$. We will need the following uniform version for infinite tuples. Say that $M$ has \df{uniform non-algebraicity} if for all (possibly infinite) tuples $\bar b$ from $M$ and $a \in M$ not from $\bar b$, there exist $\bar b', a', a''$ such that $a' \neq a''$ and $\tp (\bar b'a') = \tp (\bar b'a'') = \tp (\bar ba)$. We note that uniform non-algebraicity implies non-algebraicity. The converse holds in the presence of compactness (for example, if $G \actson M$ is oligomorphic or, more generally, if $M$ is a homogeneous structure whose age is given by a first-order universal theory and $G$ is its automorphism group).

Recall that a permutation group $G \leq \Sym(M)$ is called \df{transitive} if there is only one orbit and a transitive permutation group is called \df{primitive} if there are no non-trivial $G$-invariant partitions of $M$. A transitive action is primitive iff the stabilizer $G_a$ is a maximal subgroup of $G$ for every (some) $a \in M$. We will say that $G_a$ is \df{boundedly maximal} if there is $d \in \N$ such that for every $g \notin G_a$, every element of $G$ can be written as a word of length at most $d$ in $g, g^{-1}$, and elements of $G_a$. This is automatic for primitive actions if there are only finitely many orbits of $G$ on $M^2$ but can also happen in other situations. For a classification in the pseudo-finite case, see \cite{Liebeck2010}.

\begin{theorem}
  \label{th:tail-equal-invariant}
  Let $G \leq \Sym(M)$ be a permutation group with no algebraicity such that for every $a \in M$, the stabilizer $G_a$ is a maximal subgroup of $M$ (i.e., the action on every orbit is primitive). Suppose, moreover, that at least one of the following conditions holds:
  \begin{enumerate}
  \item \label{i:twi:uniform-noalg} $G \actson M$ has uniform non-algebraicity; or
  \item \label{i:twi:uniform-prim} for any $a \in M$, $G_a$ is boundedly maximal.
  \end{enumerate}
  Let $G \actson \Omega \coloneqq [0, 1]^M$ be the shift action and let $\mu$ be an invariant probability measure. Then $\cT = \cJ$.
  In particular, the projections $\xi_a$ are conditionally independent over $\cJ$ and if the action is ergodic, the measure $\mu$ is a product measure.
\end{theorem}
\begin{proof}
  The inclusion $\cJ \sub \cT$ follows from \autoref{p:inv-sub-tail}.

  For the other inclusion, by \autoref{p:gen-tail}, it suffices to show that for every $a \in M$, $\cF_a \cap \cT \sub \cJ$. Let $E \in \cF_a \cap \cT$. As $E \in \cT$, for every $n$, there exists a finite $B_n \sub M$ with $a \notin B_n$ and $E_n \in \cF_{B_n}$ such that $\mu(E \sdiff E_n) < 1/n$.

  Suppose now that \autoref{i:twi:uniform-noalg} holds and let the tuple $\bar b$ enumerate the set $\bigcup_n B_n$. By uniform non-algebraicity, there exist $\bar b'$ and $a' \neq a''$ with $\tp (\bar b a) = \tp (\bar b' a') = \tp (\bar b' a'')$. Let $g \in G$ be such that $g \cdot a = a'$ and let $E' = g \cdot E$. In particular, $E' \in \cF_{a'}$. Let $B_n' \sub B'$ be such that there is $h_n \in G$ satisfying $h_n \cdot B_n = B_n'$ and $h_n \cdot a = a'$. This implies that $h_n \cdot E = E'$. Let $E_n' = h_n \cdot E_n$, so that we have $\mu(E_n' \sdiff E') < 1/n$ and $E_n' \in \cF_{B_n'}$. Let $g_n \in G_{(B_n)}$ be such that $g_n \cdot a' = a''$, so that $g_n \cdot E' \in \cF_{a''}$. We have:
  \begin{equation*}
    \mu(g_n \cdot E' \sdiff E') \leq \mu(g_n \cdot E' \sdiff E_n') + \mu(E_n' \sdiff E') \leq 2/n.
  \end{equation*}
  Taking limits as $n \to \infty$, we obtain that $E' \in \cF_{a''}$. Thus $E'$ is fixed by both $G_{a'}$ and $G_{a''}$. As the orbit $G_{a'} \cdot a''$ is infinite, we must have that $\gen{G_{a'}, G_{a''}} \supsetneq G_{a'}$, so by maximality of $G_{a'}$, $E'$ is fixed by $G$. Therefore, so is $E = g^{-1} \cdot E'$.

  Suppose now that \autoref{i:twi:uniform-prim} holds and let $d$ be such that for every $h \notin G_a$, every element of $G$ can be written as a word of length at most $d$ in $h, h^{-1}$, and elements of $G_a$. Let $\eps > 0$ and let $n > (4d^2)/\eps^2$. By the no algebraicity assumption, there is $h \in G_{(B_n)}$, $h \notin G_a$. Let $E' = h \cdot E$, so that $\mu(E \sdiff E') < 2/n$. Let $\xi = \bOne_E$ and note that
  \begin{equation}
    \label{eq:hxi}
    \nm{\xi - h \cdot \xi} = \nm{\xi - h^{-1} \cdot \xi} = \mu(E \sdiff E')^{1/2} < \eps/d.
  \end{equation}
  Let $g \in G$ be arbitrary. By assumption, there exist $g_1, \ldots, g_d \in G_a$ and $\eps_1, \ldots, \eps_d \in \set{\pm 1}$ such that $g = h^{\eps_1}g_1 \cdots h^{\eps_d}g_d$. Notice that for every $\xi' \in L^2(\Omega)$ and any $g_i$,
  \begin{equation*}
    \nm{\xi - \xi'} = \nm{\xi - g_i \cdot \xi'}
  \end{equation*}
  and for any $\eps_i$, using \autoref{eq:hxi},
  \begin{equation*}
    \nm{\xi - h^{\eps_i} \cdot \xi'} \leq \nm{h^{\eps_i} \cdot \xi - h^{\eps_i} \cdot \xi'} + \nm{h^{\eps_i} \cdot \xi - \xi} \leq \nm{\xi - \xi'} + \eps/d.
  \end{equation*}
  So, by induction, $\nm{\xi - g \cdot \xi} < \eps$ and this holds for all $g \in G$. Let $\eta$ be the unique element of minimal norm in the closed convex hull of $G \cdot \xi$. Then $\nm{\xi - \eta} \leq \eps$ and $\eta$ is $G$-invariant. As $\eps$ was arbitrary, we conclude that $\xi$ is fixed by $G$, so $E \in \cJ$.

  Finally, the last claim follows from \autoref{c:tail-single}.
\end{proof}

\begin{cor}
  \label{c:primitive-oligom}
  Suppose that $G \leq \Sym(M)$ is a primitive, oligomorphic permutation group with no algebraicity. Let $G \actson [0, 1]^M$ be the shift action. Then any $G$-invariant, ergodic measure on $[0, 1]^M$ is of the form $\lambda^{\otimes M}$, where $\lambda$ is a probability measure on $[0, 1]$.
\end{cor}
\begin{proof}
  If the permutation group $G \leq \Sym(M)$ is oligomorphic, it satisfies both conditions \autoref{i:twi:uniform-noalg} and \autoref{i:twi:uniform-prim} of \autoref{th:tail-equal-invariant}. Indeed, \autoref{i:twi:uniform-noalg} follows from the no algebraicity assumption and the compactness theorem for first-order logic (by considering $M$ as an $\aleph_0$-categorical structure) and \autoref{i:twi:uniform-prim} follows from primitivity and the fact that each $G_a$ has only finitely many double cosets in $G$.
\end{proof}

\begin{remark}
  \label{rem:JT-example}
  \autoref{c:primitive-oligom} is a strengthening of \cite{Jahel2022}*{Theorem~1.1} because instead of weak elimination of imaginaries, one requires just primitivity of the action. It is claimed in \cite{Jahel2022}*{Example~5.1.3} that primitivity is not sufficient. However, the example presented there is incorrect because it does have algebraicity. To see this, observe that, in the notation of \cite{Jahel2022}*{Example~5.1.3}, for any $c, d \in S_1$, we have that $|R(c) \cap R(d)| \leq 1$. Indeed, if $a, b \in R(c) \cap R(d)$, then $c, d \in R(a) \cap R(b)$, contradicting the condition in the age. In fact, by the extension property of \Fraisse limits, $|R(c) \cap R(d)| = 1$ for all $c, d \in S_1$. This means that the algebraic closure in $S_0$ is non-trivial: take $a_1, b_1, a_2, b_2$ with $R(a_i) \cap R(b_i) = \set{c_i}$ with $c_1 \neq c_2$. Then $R(c_1) \cap R(c_2)$ is a singleton, which is in the algebraic closure of $a_1, b_1, a_2, b_2$.
\end{remark}

We end the paper with several examples of non-oligomorphic groups where \autoref{th:tail-equal-invariant} applies. All of the examples are constructed as appropriate Fraïssé limits.
\begin{example}
  \label{ex:Urysohn-Z}
  Let $\bU_\Z$ be the \df{integer Urysohn space}, i.e., the Fraïssé limit of all finite metric spaces with integer-valued metric. It is well-known that this is an amalgamation class and it is not difficult to check that the action $\Iso(\bU_\Z) \actson \bU_\Z$ satisfies condition \autoref{i:twi:uniform-noalg} and the other hypotheses of \autoref{th:tail-equal-invariant}. However, it does not satisfy condition \autoref{i:twi:uniform-prim} (and it is not oligomorphic). Indeed, let $a \in \bU_\Z$ be any point and let $g \notin G_a$. Then any element of $G$ which can be written as a word of bounded length in elements of $G_a$ and $g, g^{-1}$ can move the point $a$ only a bounded distance from itself. \autoref{th:tail-equal-invariant} for $\bU_\Z$ is not new---it was already proved in \cite{Barritault2024pp}---but this example is a good illustration for the type of situations where condition \autoref{i:twi:uniform-prim} fails.
\end{example}

\begin{example}
  \label{ex:inf-rel-no-triangles}
  Let $\cL = \set{R_i : i \in \N}$ be a signature with infinitely many binary relations. Let $\cF_1$ be the Fraïssé class of all finite $\cL$-structures such that all relations are symmetric and for every pair of points, exactly one relation holds. Let $M_1$ be its Fraïssé limit. Then the action $\Aut(M_1) \actson M_1$ is not oligomorphic (because it has infinitely many orbits on pairs) and it satisfies both conditions \autoref{i:twi:uniform-noalg} and \autoref{i:twi:uniform-prim} of \autoref{th:tail-equal-invariant}.

  Next consider the Fraïssé class $\cF_2$ in the same signature where, in addition, one forbids monochromatic triangles, i.e., configurations $\set{R_i(a, b), R_i(b, c), R_i(c, a)}$ for all $i \in \N$, and let $M_2$ denote the Fraïssé limit. Then condition \autoref{i:twi:uniform-prim} holds but \autoref{i:twi:uniform-noalg} fails. To see this, let $a \in M_2$ and let $\bar b$ be an infinite tuple such that $R_i(a, b_i)$ for all $i$. Then it is impossible to have two distinct points $a', a''$ which have the type of $a$ over a copy $\bar b'$ of $\bar b$: $R_i(a', a'')$ must hold for some $i$ and this will create a monochromatic triangle $\set{a', a'', b_i'}$.
\end{example}

\bibliography{non-singular-oligom}

\end{document}